\documentclass[11pt,reqno,a4paper]{amsart}
\usepackage[utf8]{inputenc}
\usepackage{lmodern}
\usepackage{bold-extra}
\usepackage[all]{xy}
\usepackage{dsfont}
\usepackage{hyperref}
\usepackage{color}
\usepackage{esint}
\usepackage{mathtools}
\usepackage[T1]{fontenc}
\usepackage{textcomp}
\mathtoolsset{showonlyrefs}
\usepackage{slashed}
\usepackage{bigints}
\usepackage{psfrag}
\usepackage{tikz-cd}

\usepackage{scalerel}
\usepackage{amsmath,amssymb,amsthm,amsfonts,mathrsfs,bm}

\usepackage{exscale,relsize}
\usepackage{textgreek}
\usepackage{epsfig,graphics}

\usepackage[margin=0.9in]{geometry}

\newcommand{\Abracket}[1]{<#1>} 
\newcommand{\parenthesis}[1]{\left(#1\right)} 
\newcommand{\braces}[1]{\left\{#1\right\}} 

\newcommand{\R}{\mathbb{R}}

\newcommand{\dd}{\mathop{}\!\mathrm{d}}
\newcommand{\eps}{\varepsilon}

\newcommand{\ii}{\infty}
\newcommand{\graf}[1]{\left\{\begin{array}{ll}#1\end{array}\right.}

\newcommand{\dt}{\delta}

\newcommand{\al}{\alpha}
\newcommand{\be}{\beta}

\newcommand{\sg}{\sigma}

\newcommand{\om}{\Omega}
\newcommand{\lm}{\lambda}

\newcommand{\pa}{\partial}

\newcommand{\prl}{{\textbf{(}\mathbf P\textbf{)}_{\mathbf \lm}}}

\newcommand{\vxi}{\xi}

\newcommand{\sscp}{\scriptscriptstyle}
\newcommand{\dsp}{\displaystyle}
\newcommand{\noi}{\noindent}

\newcommand{\ssl}{\sscp \lm}
\newcommand{\ml}{m_{\sscp \lm}}

\newcommand{\all}{\al_{\ssl}}

\newcommand{\el}{E_{\ssl}}

\newcommand{\tl}{\tau_{\ssl}}

\newcommand{\pl}{\psi_{\sscp \lm}}

\newcommand{\rife}[1]{(\ref{#1})}
\newcommand{\ov}[1]{\overline{#1}}

\newcommand{\rl}{\mbox{\Large \textrho}_{\!\sscp \lm}}

\newcommand{\rlq}{V_{\lm}}

\newcommand{\ino}{\int_{\Omega}}

\newcommand{\inpo}{\int_{\pa \Omega}}

\DeclareMathOperator{\Eigen}{Eigen}

\DeclareMathOperator{\Span}{Span}

\DeclareMathOperator{\Spect}{Spect}

\newtheorem{thm}{Theorem}[section]

\newtheorem{lemma}[thm]{Lemma}
\newtheorem{prop}[thm]{Proposition}

\newtheorem{rmk}[thm]{Remark}

\title[Sharp spectral estimates]{Sharp spectral estimates for free boundary problems arising in plasma physics}

\thanks{2020 \textit{Mathematics Subject classification:} 35B09, 35B32, 35J61, 35Q99, 35R35, 82D10.}

\author[D. Bartolucci]{Daniele Bartolucci}
\address{Daniele Bartolucci, Department of Mathematics, University of Rome \emph{"Tor Vergata"}, Via della ricerca scientifica n.1, 00133 Roma.}
\email{bartoluc@mat.uniroma2.it}

\author[A. Jevnikar]{Aleks Jevnikar}
\address{Aleks Jevnikar, Department of Mathematics, Computer Science and Physics, University of Udine, Via delle Scienze 206, 33100 Udine, Italy.}
\email{aleks.jevnikar@uniud.it}

\author[J. Wei]{Juncheng Wei}
\address{Juncheng Wei, Department of Mathematics, Chinese University of Hong Kong, Shatin, Hong Kong.}
\email{wei@math.cuhk.edu.hk}

\author[R. Wu]{Ruijun Wu}
\address{Ruijun Wu, School of mathematics and statistics, Beijing Institute of Technology, Zhongguancun South Street No. 5, 100081 Beijing, P.R.China.}
\email{ruijun.wu@bit.edu.cn}

\thanks{ D.B.
is partially supported by the MIUR Excellence Department Project MatMod@TOV
awarded to the Department of Mathematics, University of Rome ``Tor Vergata'', by PRIN project 2022, ERC PE1\_11,
``{\em Variational and Analytical aspects of Geometric PDEs}'', by INdAM-GNAMPA project ``{\em Analisi qualitativa di problemi differenziali non lineari}'', by the E.P.G.P. Project sponsored by the University of Rome ``Tor Vergata'' and by INdAM-GNAMPA project ``{\em Analisi qualitativa di problemi differenziali non lineari}''.\\
 \indent A.J. is partially supported by INdAM-GNAMPA project ``{\em Analisi qualitativa di problemi differenziali non lineari}'' and PRIN Project 20227HX33Z ``{\em Pattern formation in nonlinear phenomena}''.\\
\indent J.W. is partially supported by GRF fund of RGC of Hong Kong entitled
``\emph{New frontiers in singularity formations of nonlinear partial differential equations}''.
}

\thanks{We would like to express our warmest thanks to P. Korman for very fruitful exchange of ideas about global bifurcation of Gelfand problems.}
\begin{document}
\numberwithin{equation}{section}

\begin{abstract}
    We derive a sharp spectral estimate for a superlinear free boundary problem arising in plasma physics. The semilinear equation is coupled with a constraint, which forces the analysis of a non-local eigenvalue equation. Consequently the corresponding first eigenvalue, say $\sg_1$, is not a standard one and it is shown that it cannot satisfy a general isoperimetric property of Faber-Krahn type. This motivates a careful analysis of the problem on balls in any dimension $N\geq 2$, where we prove that in fact $\sg_1$ is always positive. The implications about the uniqueness problem for the Emden equation are also discussed.
\end{abstract}

\maketitle
{\bf Keywords}: Spectral estimates, Free boundary equations, Emden problems, Uniqueness.

\section{Introduction}
Let~$\Omega\subset \R^N$,~$N\ge 2$ be a bounded domain of class~$C^{3}$ and set
$$
p_{_N}=\graf{+\ii,\;N=2 \\ \frac{N}{N-2},\; N\geq 3.}
$$
Motivated by the analysis of Gelfand problems we have recently pursued (\cite{BJWW2})
the constrained problem,
$$
\graf{-\Delta \psi =[\al+{\lm}\psi]_+^p\quad \mbox{in}\;\;\om\\ \\
\bigintss\limits_{\om} { [\al+{\lm}\psi]_+^p}=1\\ \\
\psi=0 \quad \mbox{on}\;\;\pa\om
}\qquad \prl
$$
for $1<p<p_{_N}$, $\lm\geq 0$ and the unknowns $\al\in\R$ and $\psi \in C^{2,r}_{0}(\ov{\om}\,)$, $r\in(0,1)$.
The underlying idea of \cite{BJ2022new, BJWW2, BJW2024sharp} is to gather enough information to follow
“the” curve of solutions of $\prl$ from $\lm=0$ and $\all=\al_0=1$, while keeping as far as possible in particular
the monotonicity of a suitably defined energy.

\bigskip

Let $q$ be the H\"older conjugate of~$p$, i.e.~$\frac{1}{p}+\frac{1}{q}=1$.
For a fixed $(\al_\lm,\psi_\lm)$ which solves $\prl$, we will adopt the following notations,
$$
 \rl= [\al_\lm+\lm \psi_\lm]_+^p \quad\mbox{\rm and }\quad {\rlq} =[\al_\lm+\lm \psi_\lm]_+^{p-1}=\rl^{\frac{1}{q}}.
$$
Problem $\prl$ is of independent interest due to its relevance to Tokamak's plasma physics (\cite{Freidberg2014IdealMHD, Kadomtsev1996nonlinear, Stacey2012fusion}),
which motivated the lot of work to understand existence, uniqueness, multiplicity of solutions and existence/non-existence/structure of the free boundary $\pa\{x\in\om\,:\,\al+\lm \psi>0\}$ of $\prl$,
see e.g.~\cite{BandleMarcus1982boundary,BandleSperb1983qualitative,BJ2022uniqueness,BJW2024sharp, Berestycki1980free, CaffarelliFriedman1980asymptotic, CaoPengYan2010multiplicity,FriedmanLiu1995free,LiPeng2015multipeak,Liu2014multiple,SuzukiTakahashi2016critical,Wei2001multiple,Wolansky1996critical} and references quoted therein.
In particular due to the fundamental results in \cite{Berestycki1980free}, it is well known that for any~$\lambda>0$ there exists at least one solution of~$\prl$, see also Section \ref{var-sol} below. The variational formulation adopted in \cite{Berestycki1980free} is closely related to minimal free energy principles with non-extensive entropy, where $-\lm$ plays the role of the inverse statistical temperature, see Appendix A in \cite{BJW2024sharp} and \cite{CLMP1995special} for further details.

On the other side, due to the constraint in $\prl$, we miss the classical notion of minimal solutions (\cite{CrandallRabinowitz1975continuation}) for this problem.
For example it is not true in general that any first eigenfunction of the associated linearized operator (see \rife{lineq0.1} below) must be simple and neither that if the first eigenvalue (say $\sg_1$) is positive then the maximum principle holds, see \cite{BJ2021global} for an example illustrating this point and
\cite{BJW2024spec} for the adaptation of the Courant nodal domain Theorem on balls.
As a consequence, among other things, unlike classical branches of Gelfand problems (\cite{CrandallRabinowitz1975continuation}) it is not at all clear that $\rl=[\all+\lm \pl]_+^p$ is pointwise
monotonic increasing as far as $\sg_1>0$.
After \cite{BJ2022new,BJ2022uniqueness,BJW2024sharp}, a workaround to this problem has been recently found in \cite{BJWW2}, yielding suitable monotonicity property through the bifurcation analysis of $\prl$, see Theorem C.

\bigskip

We follow the traditional convention to assume that the domain has unit volume, $|\Omega|=1$.
In particular $\mathbb{D}_{N}$ denotes a ball in $\R^N$ centered at the origin of unit volume $|\mathbb{D}_N|=1$, whose radius is denoted by~$R_N$.
For fixed $t\geq 1$ we define
$$
\Lambda(\om,t)=\inf\limits_{w\in H^1_0(\om), w\equiv \!\!\!\!/ \;0}
\dfrac{\ino |\nabla w|^2}{\left(\ino |w|^{t}\right)^{\frac2t}}\,,
$$
which is related to the best constant in the Sobolev embedding $\|w\|_t\leq \mathcal{C}_S(\om,t)\|\nabla w\|_2$ via
\begin{align}\nonumber
 \mathcal{C}_S(\om,t)=\Lambda^{-1/2}(\om,t) \quad \mbox{ for } t\in[1,2p_{_N}).
\end{align}

\noi {\bf Definition.} {\it We say that a solution $(\all,\pl)$ of {\rm $\prl$} is {\bf positive}
{\rm[}resp. {\bf non-negative}{\rm]} if~$\all> 0$ {\rm[}resp. $\all\geq 0${\rm]}}.\\

 It has been recently proved in \cite{BJW2024sharp} that the following quantity is well-defined and strictly positive:
$$
\lm_+(\om,p)=
\sup\left\{\delta>0\,:\, \all>0 \mbox{ for any solution } (\all,\pl)
\mbox{ of } \prl \mbox{ with }\lm<\delta\right\}.
$$

\noi {\bf Remark.} {\it We will use the fact that $\lm_+(\om,p)$ is always finite, as readily follows by the variational formulation of solutions of {\rm $\prl$}. 
See also{\rm ~\cite[Appendix A]{BJW2024sharp}}.}\\

\noi {\bf Definition.} {\it Let $(\all,\pl)$ be a solution of {\rm $\prl$}.
The {\bf energy} of $(\all,\pl)$ is
$$
\el=\frac12\ino |\nabla \pl|^2=\frac12 \ino  \mbox{\rm$\rl$}\pl.
$$
}

\noi Let us define
$$
\lm_0(\om,p):=\frac{1}{p}\Lambda(\om,2p),\quad  \om\subset \R^N,\; N\geq 2,
$$
and
$$
\lm_1(\om,p):=\left(\frac{8\pi}{p+1}\right)^{\frac{p-1}{2p}}\Lambda^{\frac{p+1}{2p}}(\om,p+1),\quad \om\subset \R^2.
$$

\noi The following results about positive solutions have been recently proved in~\cite{BJ2022uniqueness,BJW2024sharp}. \\

\noi{\bf Theorem A}{(\cite{BJ2022uniqueness, BJW2024sharp})}. {\it Let $N=2$ and $p\in [1,+\ii)$, then $\lm_+(\om,p)\geq \lm_1(\om,p)$
where the equality holds if and only if either $p=1$ or $\om=\mathbb{D}_2$.
Moreover, for any $\lm<\min\{\lm_0(\om,p),\lm_1(\om,p)\}$ there exists a unique solution to {\rm $\prl$},
defining a
real analytic curve of positive solutions which we denote by~$\mathcal{G}_0(\om)$,
such that $\al_0=1$, $2 E_0=2 E_0(\om)$ is the torsional rigidity of $\om$ and
$$
\frac{d \all}{d\lm}<0,\quad \frac{d \el}{d\lm}>0,\quad \forall (\all,\pl)\in\mathcal{G}_0(\om).
$$}

\noi We tend to believe that $\lm_1(\om,p)>\lm_0(\om,p)$, which is always true either for $\om=\mathbb{D}_2$ or for any $p$ large enough, see \cite{BJW2024sharp}.
It is not clear whether or not the sharp positivity estimates about $\lm_+(\om,p)$ obtained in Theorem B can be extended to higher dimensions $N\geq 3$.
For later purposes we also recall the following result in \cite{BJW2024sharp}. Let~$u_0$ be the unique (\cite{GNN1979symmetry}) solution of the Emden equation
\begin{align}\label{Emden0}
    \graf{-\Delta u_0 =u_0^p\quad \mbox{in}\;\;B_1\\
u_0>0 \quad \mbox{in}\;\;B_1\\
u_0=0 \quad \mbox{on}\;\;\pa B_1,
}
\end{align}
and set $$
I_p=\int\limits_{B_{1}} {\dsp {u}_0^p}.
$$
Here $R_N$ denotes the radius of $\mathbb{D}_{N}$ and it has been shown in \cite{BJW2024sharp} that,
\begin{align}\label{lmpN}
 \lm_+(\mathbb{D}_{N},p)={I_p^{1-\frac{1}{p}}}{R_N^{-\frac{N}{p}(1-\frac{p}{p_{_N}})}}.
\end{align}

\noi
Remark that for $N=2$ the uniqueness of solutions on balls was first proved in \cite{BandleSperb1983qualitative}. This being said, we have,\\

\noi{\bf Theorem B}{(\cite{BJWW2},\cite{BJW2024sharp})}. {\it Let $\om=\mathbb{D}_{N}\subset \R^N$, $N\geq 2$ and $p\in (1,p_{_N})$. Then, for any $\lm\geq 0$,
{\rm $\prl$} admits a unique solution. Moreover
$\all>0$ if and only if $\lm<\lm_+(\mathbb{D}_N,p)$, $\all$ is of class $C^2$ in $(0,+\ii)$ and $\frac{d \all}{d\lm}<0$ for any $\lm\geq 0$.}

\bigskip

\noi As a corollary of Theorem B, we have a new proof of the well known (\cite{GNN1979symmetry}) uniqueness of solutions of the Emden problem \eqref{Emden0}, see Section \ref{sec:Emden}.\\

\noi It can be shown (see \cite{BJ2022uniqueness}, \cite{BJWW2}), that the “natural” non-local
eigenvalue problem associated to the linearized equation for $\prl$ evaluated at a fixed solution $(\all,\pl)$ takes the form
$$
-\Delta \phi-\tl \rlq [\phi]_{\ssl}=\sg\rlq [\phi]_{\ssl},
$$
where, here and in the sequel, $\sg\in \R$ denotes an eigenvalue relative to $\phi$,
$$
\tl=\lm p,
$$
and
$$
\Abracket{\phi}_{\ssl}\coloneqq \dfrac{\ino \rlq  \phi}{\ino \rlq},
\qquad
[\phi]_{\ssl}\coloneqq  \phi \,- \Abracket{\phi}_{\ssl},\qquad \forall\,\phi \in L^2(\om).
$$
For any fixed solution $(\all,\pl)$, the weighted first eigenvalue is well defined (\cite{BJWW2}), which we will always denote by $\sg_1(\all,\pl)$.
Let us define
$$
\lm_*(\om,p)\coloneqq
\sup\left\{\ov{\lm}>0\,:\, \sg_1(\all,\pl)>0 \mbox{ for any solution } (\all,\pl)
\mbox{ of } \prl \mbox{ with }\lm<\ov{\lm}\right\}.
$$

\noi We recently proved the following  \\
\noi{\bf Theorem C}{(\cite{BJWW2})} {\it
 Let $N\geq 2$ and $p\in (1,p_N)$, then $\lm_*(\om,p)>\lm_0(\om,p)$. Moreover $\lm_+(\om,p)\geq \lm_0(\om,p)$ where the equality holds if and only if $p=1$
 and for any $\lm\in [0,\lm_*(\om,p))$, there exists a unique solution to {\rm $\prl$},
defining a $C^1$ curve of solutions, denoted by $\mathcal{G}_*(\om)$, such that $\al_0=1$, $2 E_0=2 E_0(\om)$ is the torsional rigidity of $\om$ and}
$$
\frac{d \all}{d\lm}<0,\quad \frac{d \el}{d\lm}>0,\quad \forall (\all,\pl)\in\mathcal{G}_*(\om).
$$

\noi As a corollary of Theorem C we obtained in \cite{BJWW2} a new way to describe the sharp qualitative behavior of the unbounded Rabinowitz continuum (\cite{Rabinowitz1971some}) of a class of Gelfand problems, including non minimal (\cite{CrandallRabinowitz1975continuation}) solutions. Indeed we reduced the problem of the existence of a {global strictly monotone} (in terms of $\el$) parametrization to an estimate about $\lm_*(\om,p)$ and $\lm_+(\om,p)$. Sufficient conditions to come up with this refined description are either
  $$
  \lm_*(\om,p)>\lm_+(\om,p),
  $$
  or
  $$
  \lm_*(\om,p)\geq \lm_+(\om,p)\quad \mbox{and}\quad \all\searrow 0^+\mbox{ as }\lm\nearrow \lm_+(\om,p),
  $$
  see \cite{BJWW2} for further details. Actually we proved in \cite{BJWW2} that
$\lm_*(\mathbb{D}_N,p)\geq \lm_+(\mathbb{D}_N,p)$ and that $\all\searrow 0^+$ as $\lm\nearrow  \lm_+(\mathbb{D}_N,p)$. As far as semilinear problems are concerned, spectral estimates of this sort are rather hard to establish,
known results generally relying on isoperimetric properties inherited by the underlying geometric structures, see \cite{BartolucciLin2014MathAnn}, \cite{ChangChenLin2003}, \cite{Suzuki1992AIHP} and references quoted therein. We come up here with a sharp refinement of the spectral estimate on $\mathbb{D}_N$. Indeed we have,
\begin{thm}\label{thm:lm*} Let $p\in (1,p_N)$, then
$$
  \lm_*(\mathbb{D}_N,p)=+\ii,
$$
that is, for any $\lm\geq 0$ and for any solution $(\all,\pl)$ of {\mbox{\rm $\prl$}} on the ball~$\mathbb{D}_N$ the inequality $\sg_1(\all,\pl)>0$ holds true.
\end{thm}

We hope this could be useful for later developments, the underlying idea being that to compare the value of $\sg_1(\all,\pl)$ for a fixed $\lm$
on a given domain $\om$ with that on $\mathbb{D}_N$. However there is no chance to rely on general
isoperimetric properties of Faber-Krahn type for $\sg_1$. Indeed, if one could prove that $\lm_*(\om,p)=+\ii$, then by Theorem C we could infer the uniqueness of solutions of the Emden problem \eqref{Emden0} on $\om$, see Section \ref{sec:Emden} for details. Therefore we know for sure that $\lm_*(\om,p)<+\ii$
for all those domains where uniqueness of the Emden problem fails, see for example \cite{EMP2}, \cite{LWZ} and references quoted therein.\\
This is in striking contrast with the analogue situation arising in the analysis of the first constrained eigenvalue (say again $\sg_1$) for mean field equations (\cite{BartolucciLin2014MathAnn}, \cite{ChangChenLin2003}, \cite{Suzuki1992AIHP}) where neither the geometry nor the topology of the domain affect the positivity  of $\sg_1$.\\
Actually the proof of Theorem \ref{thm:lm*} heavily relies on the uniqueness and regularity (with respect to $\lm$) of solutions on $\mathbb{D}_N$ (see Theorem B) which is not easy to extend to more general situations. This motivates the following

\smallskip

\noi{\bf OPEN PROBLEM} Assume $N\geq 2$, $1<p<p_{_N}$ and $\om$ convex, is it true that $\lm_*(\om,p)=+\ii$?

\smallskip

At last, by Propositions \ref{prop:simple} and \ref{prop:tranverse} below, if $\om$ is a two dimensional domain, symmetric and convex
with respect to coordinate directions, then we can prove that the curve of unique solutions of $\prl$ determined in Theorems A and C can be continued to a global (i.e. defined for any $\lm>0$) simple curve without bifurcation points. We conjecture that this curve contains in fact all the solutions of $\prl$, that
 $\all$ and $\el$ are monotonic all along the branch, and in particular that $\lm_*(\om,p)=+\ii$ for these domains.
Interestingly enough, in the same spirit of \cite{BJWW2}, this would yield a new description of the global solutions branch of Gelfand problems first pushed forward in \cite{Holzman1994uniqueness}.

\section{Preliminary results}
We will use several well known facts, see \cite{BJWW2}.
\begin{lemma}\label{lemE1}
    Let $p\in[1,p_{_N})$. For any $\ov{\lm}>0$ there exist $\ov{\al}=\ov{\al}(r,\om,\ov{\lm},p,N)>-\ii$ and $C_1=C_1(r,\om,\ov{\lm},p,N)<+\infty$ such that
    \begin{align}
        \all\geq \ov{\al}, \qquad \|\pl\|_{C^{2,r}_0(\ov{\om})}\leq C_1
    \end{align}
    for any solution~$(\all,\pl)$ of {\rm $\prl$} with $\lm\in [0,\ov{\lm}\,]$.
\end{lemma}

For fixed $\lm\geq 0$ and $p\in(1,p_{_N})$, let $(\all,\pl)$ be a solution of $\prl$, we recall the following notation from the introduction,
$$
\Abracket{\phi}_{\ssl}=\dfrac{\ino \rlq  \phi}{\ino \rlq}\quad \mbox{and }\quad
[\phi]_{\ssl}=\phi \,-\Abracket{\phi}_{\ssl},\quad \phi \in L^2(\om).
$$
Let us recall that we are always concerned with classical solutions and $p>1$, whence in particular
\begin{align}
    \rlq=[\all+\lm\pl]_+^{p-1}\in C^r(\ov{\om}),
\end{align}
for some $r\in (0,1)$. The support of $\rl$ is by definition $\ov{\om_+}$,
whence, if~$\all\geq 0$ then~$\om_+= \om$ while if~$\all<0$ then, since $\pl$ is continuous, we have $\om_+\Subset \om$.\\
\noi Consider the linear operator~$L_{\ssl}\colon C^{2,r}_0(\ov{\Omega})\to C^r(\ov{\Omega})$ defined by
\begin{align}\label{eLl}
    L_{\ssl}[\phi]=-\Delta \phi-\tl \rlq [\phi]_{\ssl}
\end{align}
with~$\tl=\lm p$.
We say that $\sg=\sg(\all,\pl)\in\R$ is an eigenvalue of $L_{\ssl}$ if the equation
\begin{align}\label{lineq0.1}
-\Delta \phi = (\tl+\sigma)\rlq [\phi]_{\ssl},
\end{align}
admits a non-trivial weak solution $\phi\in H^1_0(\om)$, and the corresponding eigenspace is denoted by $\Eigen(L_{\ssl};\sg)$.
Note that~$L_{\ssl}$ involves not only nonlocal terms but also a weight which may vanish on a large set. This situation has been
discussed in \cite{BJWW2} after \cite{BJ2022uniqueness} as we shortly illustrate below.\\
If we let~$G$ denote the Green's function for~$-\Delta$ with Dirichlet boundary condition
and apply it as a convolution kernel to~\eqref{lineq0.1}, then we see that,
\begin{align}
    \phi =  \frac{\tl+\sg}{\tl}\;  G* \parenthesis{\tl\rlq[\phi]_{\ssl}}.
\end{align}
We consider the operator~$T_{\ssl} \colon H^1_0(\Omega)\to H^1_0(\Omega)$, to be defined as follows,
\begin{align}
    T_{\ssl}(\phi) \coloneqq G* \parenthesis{\tl\rlq[\phi]_{\ssl}},
\end{align}
so that~\eqref{lineq0.1} is equivalent to
\begin{align}\label{eq:eigenvalue for T}
    T_{\ssl}(\phi)= \mu\phi, \qquad \mbox{ with } \mu = \frac{\tl}{\tl+\sg}.
\end{align}
In other words,~$\phi\in \Eigen(L_{\ssl};\sigma)$ iff~$\phi\in \Eigen(T_{\ssl};\frac{\tl}{\tl+\sg})$.
Thus it suffices to understand the eigenvalues and eigenfunctions of~$T_{\ssl}$, where the advantage is that it is a
linear self-adjoint compact operator on the Hilbert space~$H^1_0(\Omega)$ equipped with the inner product
\begin{align}
    \Abracket{\xi,\eta}_{H^1_0}=\int_{\Omega} \Abracket{\nabla\xi,\; \nabla\eta}, \qquad  \forall \xi,\eta\in H^1_0(\Omega),
\end{align}
see \cite{BJWW2} for details.
As a consequence, the spectrum of~$T_{\ssl}$,~$\Spect(T_{\ssl})$, consists of countably many real nonzero eigenvalues and perhaps also zero.
Each nonzero eigenvalue has finite multiplicity, and the nonzero eigenvalues have zero as the unique accumulation point.
From~\eqref{lineq0.1} it follows that~$\tl+\sigma\geq 0$, hence~$\Spect(T_{\ssl})\subset\R_+$.
We may thus list the nonzero eigenvalues as follows,
\begin{align}
    \mu_1\geq \mu_2\geq \mu_3\geq \cdots >0, \quad \lim_{j\to +\infty}\mu_j = 0,
\end{align}
and denote the corresponding eigenfunctions by~$\phi_j$,~$j\in\mathbb{N}$, which satisfy
\begin{align}
    \delta_{jk}=\Abracket{\phi_j,\phi_k}_{H^1_0}=\int_\Omega \Abracket{\nabla \phi_j, \;\nabla\phi_k}, \qquad \forall j,k\geq 1.
\end{align}
We collect the set of nonzero eigenvalues in~$\Spect(T_{\ssl})\setminus \braces{0}$ and denote
\begin{align}
    \mathcal{H}_1\coloneqq \ov{\oplus_{j\geq 1} \Eigen(T_{\ssl};\mu_j) } = \ov{\Span\braces{\phi_j\mid j\geq 1}}
\end{align}
where the closure is taken in the~$H^1_0$ norm. This is a closed subspace of~$H^1_0(\Omega)$.
\begin{lemma}{{\rm (}\cite{BJWW2}{\rm )}}\label{lem:spectral}$\,$
    \begin{itemize}
        \item[(i)] $0\in\Spect(T_{\ssl})$ iff~$\alpha<0$.
        \item[(ii)] If~$\alpha\geq 0$, then
                    \begin{align}
                        \Spect(T_{\ssl})=\braces{\mu_j\mid j\geq 1}\quad \mbox{ and }\quad
                        H^1_0(\Omega)=\mathcal{H}_1.
                    \end{align}
        \item[(iii)] If~$\alpha<0$, then
                \begin{align}
                    \Spect(T_{\ssl})=\braces{0}\cup \braces{\mu_j\mid j\geq 1}\quad  \mbox{ and }\quad
                    H^1_0(\Omega)= \Eigen(T_{\ssl};0)\oplus\mathcal{H}_1.
                \end{align}
                Moreover, the restriction of any function in~$\mathcal{H}_1$ to the set~$\Omega\setminus\ov{\Omega_+}$ is harmonic,
                \begin{align}\label{eq:og decomp}
                    \Eigen(T_{\ssl};0)=\braces{\phi\in H^1_0(\Omega)\mid \phi|_{\Omega_+} = \, <\phi>_{\lm} \mbox{a.e. in }\Omega_+},
                \end{align}
                and
                $$
                \phi\in \Eigen(T_{\ssl};0) \Rightarrow {\mbox{\rm $\rlq$}}[\phi]_{\ssl} = 0 \mbox{ a.e. in }\om.
                $$
    \end{itemize}
\end{lemma}

\begin{rmk}\label{rem:orth}
    In the orthogonal decomposition, we will use
    \begin{align}
        P_0\colon H^1_0(\Omega)\to\Eigen(T_{\ssl};0) \quad  \mbox{ and }\quad
        P_1\colon H^1_0(\Omega)\to \mathcal{H}_1
    \end{align}
    to denote the corresponding orthogonal projections.
    More precisely, for any~$\psi\in H^1_0(\Omega)$, we have that
    \begin{align}
        \psi= \psi_0+\psi_1=P_0 \psi + P_1 \psi = P_0\psi + \sum_{j=1}^{+\infty} \beta_j\phi_j
    \end{align}
    with the Fourier coefficients {\rm
    \begin{align}
        \beta_j
        = \Abracket{\psi,\phi_j}_{H^1_0}
        =\ino \Abracket{\nabla\psi,\;\nabla\phi_j}
        =\ino \frac{\tl}{\mu_j} \rlq [\phi_j]_{\ssl} \psi
        =\ino \frac{\tl}{\mu_j} \rlq [\phi_j]_{\ssl} [\psi]_{\ssl}.
    \end{align}}

\end{rmk}

Recall that each~$\mu_j>0$ of~$T_{\ssl}$ corresponds to an eigenvalue~$\sg_j$ of~$L_{\ssl}$, related by
\begin{align}
    \mu_j=\frac{\tl}{\tl+\sg_j}, \quad \mbox{ that is, }\quad  \sg_j= \tl\parenthesis{\frac{1}{\mu_j}-1},
\end{align}
sharing the same eigenfunction~$\phi_j$. The zero eigenvalue of~$T_{\ssl}$ does not correspond to any eigenvalue of~$L_{\ssl}$.
Indeed,~$L_{\ssl}|_{\Eigen(T_{\ssl};0)}= (-\Delta)|_{\Eigen(T_{\ssl};0)}$, which gives the other part of the spectrum of~$L_{\ssl}$.\\
\noi At last remark that $0\in\Spect(L_{\ssl})$ iff~$1\in\Spect(T_{\ssl})$. Moreover we have
\begin{lemma}{{\rm (}\cite{BJWW2}{\rm )}}\label{lem:iso}
If $0\notin \Spect(L_{\ssl}) $, then~$L_{\ssl}\colon C^{2,r}_0(\ov{\Omega})\to C^r(\ov{\Omega})$ is an isomorphism.
\end{lemma}
To describe the continuous branch of solutions $\lm \mapsto (\all,\pl)$, we employ the implicit function theorem.
Thus consider the map
\begin{align}\label{eq:Phi}
    \Phi\colon (-1,+\infty)\times \R \times C^{2,r}_0(\ov{\Omega}) &\to \R \times C^r(\ov{\Omega}), \\
            (\lambda,\alpha, \psi)& \mapsto \Phi(\lambda,\alpha, \psi )=(\Phi_1, \Phi_2),
\end{align}
with
\begin{align}\label{eF}
    \Phi_1(\lambda,\alpha,\psi)=  -1+\int_\Omega [\alpha+\lambda\psi]_+^p, & &
    \Phi_2(\lambda,\alpha,\psi)= -\Delta\psi-[\alpha+\lambda\psi]_+^p  \in C^r(\ov{\Omega}).
\end{align}
The preimage~$\Phi^{-1}(0,0)$ consists exactly of solutions of~$\prl$.
The differential w.r.t. $(\alpha,\psi)$ is given by
\begin{align}
    D_{(\alpha,\psi)} \Phi(\lambda,\alpha,\psi)[s,\phi]
    = \parenthesis{ p\int_\Omega [\alpha+\lambda\psi]_+^{p-1}(s+\lambda\phi), \;  -\Delta\phi- p[\alpha+\lambda\psi]_+^{p-1}(s+\lambda\phi) }.
\end{align}

\

Fix~$\lm>0$ and let~$(\all, \pl)$ be a solution of~$\prl$, so that~$\Phi(\lm,\all,\pl)=(0,0)$.
Using the notation~$\rl$ as above, we have
\begin{align}\label{eq:differential F}
    D_{(\alpha,\psi)} \Phi(\lm,\all,\pl)[s,\phi]
    = \parenthesis{ p\int_\Omega \rlq (s+\lambda\phi), \;  -\Delta\phi-p\rlq (s+\lambda\phi)}.
\end{align}

We have the following local~$C^1$ regularity of the branch of solutions, see \cite{BJWW2}.

\begin{lemma}\label{lem1.1}
Let $(\al_{\sscp \lm_0},\psi_{\sscp \lm_0})$ be a solution of {\rm $\prl$} with $\lm=\lm_0\geq 0$.
If $0\notin\Spect(L_{\sscp \lm_0})$, then:
\begin{itemize}
    \item[(i)] $D_{(\al,\psi)}\Phi(\lm_0, \al_{\sscp \lm_0}, \psi_{\sscp \lm_0})$ is an isomorphism;
    \item[(ii)] There exists an open neighborhood $\mathcal{U}$ of $(\lm_0,\al_{\sscp \lm_0},\psi_{\sscp \lm_0})$ such that the set of solutions of {\rm $\prl$} in $\mathcal{U}$ is a $C^1$ curve of solutions $J\ni\lm\mapsto (\all,\pl)\in B$, for suitable neighborhoods $J$ of $\lm_0$ and $B$ of $(\al_{\sscp \lm_0},\psi_{\sscp \lm_0})$ in $\R\times C^{2,r}_{0}(\ov{\om}\,)$.
\end{itemize}

\end{lemma}

At last we recall the following facts from \cite{BJWW2}.
\begin{lemma}{{\rm (}\cite{BJWW2}{\rm )}}\label{lem:al}
Let $\om=\mathbb{D}_N$, $N\geq 2$, $p\in (1,p_N)$ and for any $\lm\geq 0$ let
$(\all, \pl)$ be the unique solution of {\mbox{\rm $\prl$}} as determined by Theorem B.
Then $\all\in C^{2}(0,+\ii)$, $\frac{d \all}{d\lm}<0$ in $(0,+\ii)$ and
$$
\frac{d \all}{d\lm}=<\pl>_{\ssl}+\lm<\frac{d \pl}{d\lm}>_{\ssl}
$$
\end{lemma}

\begin{prop}\label{prop:simple}
    Let $\om=B_R(0)\subset \R^N$, $N\geq 2$ or either $\om\subset \R^2$ be symmetric and convex with respect to the coordinate directions $x_i$, $i=1,2$.
    Let $(\all,\pl)$ be a solution of {\rm$\prl$} with~$\lm>0$.
Suppose that $\sg_k=\sg_k(\all,\pl)=0$ and let $\phi_k$ be any
corresponding eigenfunction. Then:
\begin{itemize}
    \item[(i)] $\Abracket{\phi_k}_{\ssl}\neq 0$;
    \item[(ii)] $\sg_k(\all,\pl)$ is simple, that is, it admits at most one linearly independent eigenfunction.
\end{itemize}
\end{prop}
\begin{proof}
The proof of Proposition \ref{prop:simple} is provided in \cite{BJWW2} for positive solutions
but it is readily seen that the same argument works fine for any solution. We report it here for the sake of completeness.\\
Concerning $(i)$ we argue by contradiction and assume that $<\phi_k>_{\ssl}= 0$.
In this case $\phi_k$ would be a classical solution of,
\begin{equation}\label{0107.1}
\graf{-\Delta \phi_k =\tl \rlq \phi_k \quad \mbox{in}\;\;\om,\\ \\
\phi_k=0 \quad \mbox{on}\;\;\pa\om,
}
\end{equation}
satisfying $\ino \rlq \phi_k=0$.
However, by Theorem 3.1 in \cite{DGP1999qualitative} (which relies for $N=2$ on the symmetry and convexity properties of $\om$) the
nodal line of any solution of \eqref{0107.1} cannot intersect the boundary. Here one is also using the fact that $\pl\geq 0$ is
a solution of $\prl$ and that
$\rl=f_{\ssl}(\pl)$, where $f_{\ssl}:[0,+\ii)\to [0,+\ii)$ is a $C^1$ function and  $f_{\ssl}(0)\geq 0$.
Therefore, $\phi_k$ has a fixed sign in a small enough neighborhood of the boundary and then, in particular,
we can assume without loss of generality that
$\pa_{\nu}\phi_k=\frac{\pa \phi_k}{\pa \nu}<0$ on $\pa \om$,
where $\nu$ denotes the exterior unit normal.
On the other side, by \eqref{0107.1} and $<\phi_k>_{\ssl}= 0$, we see that $\inpo \pa_{\nu}\phi_k=0$,
which is the desired contradiction.\\
Concerning $(ii)$, we have from $(i)$ that any eigenfunction $\phi_k$ satisfies $<\phi_{k}>_{\ssl}\neq 0$.
If there were more than one such eigenfunctions, say $\phi_{k,\ell}$, $\ell=1,2$, then, putting
$a_\ell:=<\phi_{k,\ell}>_{\ssl}\neq 0$, $\ell=1,2$, we would find that
$\phi=\phi_{k,1}-\frac{a_1}{a_2}\phi_{k,2}$ would be an eigenfunction of
$\sg_k$ satisfying $<\phi>_{\ssl}=0$, which is a contradiction.
\end{proof}

\section{Variational Solutions}\label{var-sol}
\noi We discuss a spectral estimate for variational solutions. In fact, solutions of $\prl$ can be found (\cite{Berestycki1980free})
minimizing the free energy $\mathcal{F}_{\lm}(\rho)$ which, takes the form,
\begin{equation}\label{var:1}
\inf \{\mathcal{F}_{\lm}(\rho), \rho \in \mathcal{P}\},\quad \mathcal{P}=
\left\{\rho\in L^{1+\frac1p}(\om)\,|\,\ino \rho=1,\;\; \rho\geq 0\;\mbox{a.e. in}\;\om \right\}
\end{equation}
$$
\mathcal{F}_{\lm}(\rho)=
{\scriptstyle \frac{p}{p+1}}\ino (\rho)^{1+\frac{1}{p}}-\frac\lm 2 \ino \rho G[\rho].
$$
It has been proved in \cite{Berestycki1980free} that for any $\lm$ there exists at
least one variational solution (although we are taking for granted here a slightly modified approach based on non-convex optimization principles,
see p.422 in \cite{Berestycki1980free}).\\

\noi Denoting by $\alpha$ the Lagrange multiplier relative to the mass constraint, solutions of \eqref{var:1}
satisfy the Euler-Lagrange equation,
$$
\rho^{\frac1p}=\left[\alpha+\lm G[\rho] \right]_+, \quad G[\rho]=\ino G(x,y)\rho(y)dy,
$$
which is, putting $\psi=G[\rho]$, nothing but $\prl$. Then we have,
\begin{thm}\label{prmin}
Let $(\all,\pl)$ be a variational solution of {\rm $\prl$}. Then $\sg_1(\all,\pl)\geq 0$.
\end{thm}
\begin{proof} Recall that $\rlq=[\all+\lm\pl]_+^{p-1}$.
Let $\rl$ be any minimizer of the variational problem \eqref{var:1}, we consider variations of the form,
$$
f=\rlq \varphi,\quad  \varphi\in C^{2}(\,\ov{\om}\,),
$$
such that  $\ino \rlq \varphi=0$. Then the Taylor formula shows that,
$$
0\leq J_{\ssl}(\rl+\eps f)-J_{\ssl}(\rl)=\frac{\eps^2}{2}\mathcal{D}^{2}J_\lm(\rl)[f,f\,]+\mbox{\rm o}(\eps^2),
$$
where
$$
0\leq p\mathcal{D}^{2}J_\lm(\rl)[f,f\,]=\ino (\rl)^{\frac1p-1}f^2-p\lm\ino fG[f]=\ino \rlq\varphi^2-p\lm\ino \rlq \varphi \,G[\rlq \varphi]
$$
for any $\varphi\in C^{2}(\,\ov{\om}\,)$. Remark that in principle the term $\ino \rho^{1+\frac1p}$ is not of class $C^2$, however
the expansion can be made rigorous by a standard approximation argument. Therefore, we have
$$
\mathcal{A}(\phi):=\ino \rlq [\phi]_{\ssl}^2-\tl\ino \rlq  [\phi]_{\ssl} G[\rlq [\phi]_{\ssl}]\geq 0,\quad \forall\,\phi\in C^{1}_0(\ov{\om}).
$$
On the other side, letting $\phi_1$ be a first eigenfunction whose eigenvalue is $\sg_1$, we see that
\begin{align}
 0\leq & \dfrac{\mathcal{A}(\phi_1)}{\ml}=<[\phi_1]_{\ssl}^2>_{\ssl}-\tl<[\phi_1]_{\ssl} G[\rlq [\phi_1]_{\ssl}]>_{\ssl}=<[\phi_1]_{\ssl}^2>_{\ssl}-\frac{\tl}{\tl+\sg_1}<[\phi_1]_{\ssl}\phi_1>_{\ssl} \\
 =&<[\phi_1]_{\ssl}^2>_{\ssl}\frac{\sg_1}{\tl+\sg_1}
\end{align}
and then we immediately conclude that $\sg_1\geq 0$.
\end{proof}

\begin{rmk}\label{rem:var}
By the result in \cite{Berestycki1980free} there exists at least one variational solution of {\rm $\prl$} for any $\lm$.
    Thus, in view of the uniqueness of solutions (Theorem B),
    any solution of {\rm $\prl$} in~$\mathbb{D}_{N}$ is a variational solution.
\end{rmk}

\section{The proof of Theorem \ref{thm:lm*}: \texorpdfstring{$\lm_*(\mathbb{D}_N,p)=+\ii$.}{lambda*=infty} }
\label{sec:lm*}

For later purposes we introduce the weighted product,
\begin{align}
    \Abracket{\phi,\psi}_{\ssl}\coloneqq \frac{ \ino \rlq \phi\psi}{ \ino \rlq }, \qquad \forall \phi,\psi\in L^2(\Omega).
\end{align}
Note that~$\rlq$ is continuous, hence this weighted inner product is well-defined on~$L^2(\om)$.

\

The main ingredient is the following bending lemma of Crandall-Rabinowitz type. It is a refinement of an analogous result in
\cite{BJ2022uniqueness} for positive solutions.
\begin{prop}\label{prop:tranverse}
    Let~$(\all,\pl)$ be a solution of {\rm$\prl$} with $\lm>0$ and~$0\in \Spect(L_{\ssl})$, say~$\sg_k(\all,\pl)=0$.
    Suppose that~$\sg_k(\all,\pl)=0$ is simple with the (normalized) eigenfunction~$\phi_k\in C^{2,r}_0(\ov{\om}\,)$ satisfying~$\Abracket{\phi_k}_{\ssl}\neq 0$. Then,
    \begin{align}\label{2907.5}
                    \Abracket{[\phi_k]_{\ssl} \pl}_{\ssl}\neq 0\mbox{ and } \Abracket{[\phi_k]_{\ssl} \pl}_{\ssl} \mbox{   has the same sign as } \Abracket{\phi_k}_{\ssl},
    \end{align}
and there exist $\eps>0$, an open neighborhood $\mathcal{U}$ of $(\lm,\all,\pl)$ in $(0,+\ii)\times(0,1)\times C^{2,r}(\ov{\om})$ and a $C^1$ curve
    \begin{align}
        (-\eps,\eps) \to \mathcal{U}, \qquad s\mapsto (\lm(s), \al(s),\psi(s)),
    \end{align}
    such that
    \begin{itemize}
        \item $(\lm(0), \al(0),\psi(0))=(\lm,\all,\pl)$,
        \item $\Phi^{-1}(0,0)\cap\mathcal{U}= \braces{(\lm(s), \al(s),\psi(s)) \mid -\eps < s< \eps }$, where~$\Phi$ is given in~\eqref{eq:Phi}.
    \end{itemize}
    Furthermore, locally near the given solution~$(\all,\pl)$, we have~$\psi(s)=\pl+s\phi_k+\vxi(s)$, with
                \begin{align}\label{2907.0}
                    \Abracket{ [ \phi_k]_{\sscp \lm(s)},\vxi(s) }_{\sscp \lm(s)}=0,\quad s\in (-\eps,\eps)
                \end{align}
                and~$\xi(0)=0$, while
                \begin{align}\label{2907.1}
                    \vxi^{'}(0) \equiv 0, & & 
                    \al^{'}(0)=-\lm <\phi_k>_{\ssl}, & & 
                    \lm^{'}(0)=0,& & 
                    \psi^{'}(0)= \phi_k.
                \end{align}
                Moreover,
                $\sg_k(s)$ is simple for $s\in (-\eps,\eps)$ and we have
                \begin{align}\label{2907.10}
                    \sg_k(s)\ino (\mbox{\Large\emph{\textrho}}_{\! \sscp \lm(s)})^{\frac 1q} [\phi_k(s)]_{\sscp \lm(s)}\psi^{'}(s)=p\lm^{'}(s)\ino
	(\mbox{\Large\emph{\textrho}}_{\! \sscp \lm(s)})^{\frac 1q} [\phi_k(s)]_{\sscp \lm(s)}\psi(s).
                \end{align}
\end{prop}
\begin{proof}
First of all, we prove that, since $\phi_k$ is a solution of \eqref{lineq0.1}, we have,
\begin{align}\label{2907.6n}
\frac{1}{\ml}<\phi_k>_{\ssl}\equiv(\all+\lm<\pl>_{\ssl})<\phi_k>_{\ssl}=(\lm(p-1)+\sg_k)<\pl [\phi_k]_{\ssl}>_{\ssl}.
\end{align}
\noi The left hand side equality in \eqref{2907.6n} is an immediate consequence of the following identity,
$$
(\all+\lm<\pl>_{\ssl})= \frac{1}{\ml}\left(\ino \rlq (\all+\lm\pl)\right)=\frac{1}{\ml}\left(\ino \rlq [\all+\lm\pl]_+\right)\equiv\frac{1}{\ml},
$$
where we recall that $\ml=\ino \rlq$. Therefore, we just need to prove the second equality.
By assumption $\phi_k$ satisfies $-\Delta \phi = (\tl+\sg_k)\rlq [\phi]_{\ssl}$,
which we multiply by $\pl$ and integrate by parts to obtain,
$$
\ino \rl\phi_k=(\tl+\sg_k)\ino \rlq \pl[\phi_k]_{\ssl}.
$$
\noi Dividing by $\ml$ and again since $\rl=\rlq[\all+\lm\pl]_+$ we find that,

$$
\frac{1}{\ml}\ino \rl\phi_k=\frac{1}{\ml}\ino\rlq[\all+\lm\pl]_+\phi_k= \all<\phi_k>_{\ssl}+\lm <\pl\phi_k>_{\ssl}=(\tl+\sg_k) <\pl[\phi_k]_{\ssl}>_{\ssl}.
$$
\noi The conclusion immediately follows by observing that
$$<\pl\phi_k>_{\ssl}=<\pl[\phi_k]>_{\ssl}+<\pl>_{\ssl}<\phi_k>_{\ssl},$$
which proves \eqref{2907.6n}. Therefore, since $<\phi_k>_{\ssl}\neq 0$ by assumption, we infer from \eqref{2907.6n} that \eqref{2907.5} is satisfied.\\

\noi Take $\eps>0$ small,
$\dt_i=\dt_i(\eps)>0$, $i=1,2,3$, and consider the vector
$$
(\lm+\mu,\all+\beta,\pl+s\phi_k+\vxi),
$$
where
$$
(s,\mu,\beta,\vxi)\in (-\eps,\eps)\times (-\dt_1(\eps),\dt_1(\eps))\times (-\dt_2(\eps),\dt_2(\eps))\times B_{\delta_3(\eps)}^{k,\perp},
$$
$$
B_{\delta_3(\eps)}^{k,\perp}=\left\{\vxi\in C^{2,r}_0(\ov{\om}\,)\,:\,<[\phi_k]_{\ssl},\,\vxi>_{\ssl}=0,\;\|\vxi\|_{C^{2,r}_0(\ov{\om}\,)}< \dt_3(\eps)\right\},
$$

\noi Next, let us introduce the map,
$$
\Phi_0: (-\eps,\eps)\times (-\dt_1(\eps),\dt_1(\eps))\times (-\dt_2(\eps),\dt_2(\eps))\times B_{\delta_3(\eps)}^{k,\perp} \to \R \times  C^{r}(\ov{\om}\,),
$$
to be defined as follows,
\begin{align}
\Phi_0(s,\mu,\beta,\vxi)=&\Phi(\lm+\mu,\all+\beta,\pl+s\phi_k+\vxi) \\
=&\left(\begin{array}{cl} -1+\ino [\all+\beta +(\lm+\mu)(\pl+s\phi_k+\vxi)]_+^p \\
\\
 -\Delta (\pl+s\phi_k+\vxi)-[\all+\beta +(\lm+\mu)(\pl+s\phi_k+\vxi)]_+^p
\end{array}\right).
\end{align}
For $p>1$, $\Phi_0$ is locally of class $C^1$ near $(s,\mu,\beta,\vxi)=(0,0,0,0)$ and
$$
\Phi_0(0,0,0,0)=\Phi(\lm,\all,\pl)=(0,0).
$$
Let us denote
$$
X_{k}^{\perp}=\left\{\phi\in C^{2,r}_0(\ov{\om}\,)\,:\,<\phi_k,\phi>_{H_0^1(\om)}=\frac{\tl}{\mu_k}< [\phi_k]_{\ssl},\,\phi>_{\ssl}=0\right\}.
$$
Then the differential of $\Phi_0$ with respect to~$(\mu,\beta,\vxi)$ at $(0,0,0,0)$ acts on a triple $(s_\mu, s_\be,\phi)\in \R\times\R\times X_{k}^{\perp}$ as follows,
$$
D_{(\mu,\beta,\vxi)}\Phi_0(0,0,0,0)[s_\mu,s_\be,\phi]=
\left(\begin{array}{cl}
\tl \ino \rlq \phi +p s_{\mu}\ino \rlq \pl +p s_{\be}\ino \rlq \\ \\
D_\psi \Phi_2(\lm,\all,\pl)[\phi]-p\rlq \pl s_\mu-p\rlq s_\be
\end{array}\right),
$$
with $\tl=p\lm$ and $D_\psi \Phi_2(\lm,\all,\pl)[\phi]=-\Delta \phi -\tl \rlq \phi$.
The crux of the argument is to prove the following:
\begin{lemma}\label{lem.crux}
$D_{(\mu,\beta,\vxi)}\Phi_0(0,0,0,0)$ is an isomorphism of $\R\times\R\times X_{k}^{\perp}$ onto $\R\times C^{r}(\ov{\om}\,)$.
\end{lemma}
\begin{proof}
We will prove that for each $(t,f)\in \R\times C^{r}(\ov{\om}\,)$, the vectorial equation
$$
D_{(\mu,\beta,\vxi)}\Phi_0(0,0,0,0)[s_\mu,s_\be,\phi]=\left(\begin{array}{cl}t \\   f\end{array}\right),
$$
admits a unique solution $(s_\mu,s_\be,\phi)\in \R\times\R\times X_{k}^{\perp}$. From the first equation we deduce that,
$$
s_{\be}=\frac{t }{p\ml}-<\pl>_{\ssl} s_\mu -\lm  <\phi>_{\ssl},
$$
where we recall that $\ml=\ino\rlq$.
Therefore, by substituting into the second equation, we find that the pair $(s_\mu,\phi)$ must solve the equation,
\begin{align}\label{26.7.1}
-\Delta \phi-\tl \rlq [\phi]_{\ssl}-p\rlq [\pl]_{\ssl} s_\mu = f+ t \frac{\rlq}{\ml}.
\end{align}

Let us write $C^r(\ov{\om}\,)=Y_k \oplus R$, where $Y_k =\mbox{Span}\{\phi_k\}$ and
$$
R=\left\{\phi\in C^{r}(\ov{\om}\,)\,:\,<\phi_k,\,\phi>_{L^2}=\ino \phi_k\, \phi=0 \right\},
$$
where in the image space $C^{r}(\ov{\om}\,)\subset L^2(\om)$ we use the standard scalar product $<f,g>_{L^2}=\ino f\, g$.\\
\noi Multiplying \rife{26.7.1} by $\phi_k$ and integrating we find that,
\begin{align}\label{26.7.2}
-m_{\ssl}p <[\pl]_{\ssl}, \phi_k>_{\ssl}s_\mu = \ino f\phi_k+t <\phi_k>_{\ssl},
\end{align}
where we used the fact that, since $\phi_k$ satisfies \eqref{lineq0.1} with $\sg_k=0$ (or using that~$\phi\in X_{k}^{\perp}$), then
\begin{align}\label{26.7.2-26}
<\phi_k,\, -\Delta \phi-\tl \rlq [\phi]_{\ssl}>_{L^2}
=&\ino \phi_k(-\Delta \phi-\tl \rlq [\phi]_{\ssl}) \\
=&\ino (-\Delta \phi_k-\tl \rlq [\phi_k]_{\ssl})\phi
=\sg_k\ino  \rlq [\phi_k]_{\ssl} \phi =0.
\end{align}
Thus, we see from \eqref{2907.5} that \eqref{26.7.2} admits a unique solution $s_\mu=s_\mu(t,f)$, such that \eqref{26.7.2} is satisfied,
that is
\begin{align}\label{26.7.2-26.3}
\hspace{-1cm}<\phi_k,\,p\rlq [\pl]_{\ssl} s_\mu(t,f) + f+ t \frac{\rlq}{\ml}>_{R}=
\ino \phi_k (p\rlq [\pl]_{\ssl} s_\mu(t,f) + f+ t \frac{\rlq}{\ml})=0,
\end{align}
and we are left with showing that the projection of \eqref{26.7.1} onto $R$ in the sense of the scalar product $<\cdot,\cdot>_{L^2}$ admits
a unique solution in $X_k^{\perp}$.
Observe that, again in view of \eqref{26.7.2-26} and \eqref{26.7.2-26.3},
the projection of \eqref{26.7.1} onto $R$ takes the form
\begin{align}\label{eq:IFT eqn}
L_\lm(\phi)=-\Delta \phi-\tl \rlq [\phi]_{\ssl}= g,
\end{align}
with
$$
g:=P_R\left(p\rlq [\pl]_{\ssl} s_\mu(t,f)+f+  t \frac{\rlq}{\ml}\right)
=p\rlq [\pl]_{\ssl}s_\mu(t,f)+f+  t \frac{\rlq}{\ml},
$$
where $P_R$ denotes the projection operator.
It remains to show that~\eqref{eq:IFT eqn} admits a unique solution~$\phi$ in~$X_k^\perp$.

Note that~$\ker(L_\lm)=\Eigen(L_\lm;0)=\Eigen(T_\lm;1)$.
Applying the Green function to both sides of the equation, we get
\begin{align}
 \phi-T_\lm(\phi)=G* g.
\end{align}
If~$\all> 0$ this is uniquely solvable as shown in~\cite{BJ2022uniqueness}, 
while if~$\all<0$ the situation is different, since $\Eigen(L_\lm;0)$ is not empty.
We discuss this case hereafter. It will be clear from the proof that this argument works fine 
for the borderline case~$\all=0$ as well.\\

Decompose both sides according to~$H^1_0(\Omega)=\Eigen(T_\lm;0)\oplus\mathcal{H}_1$, say
\begin{align}
 \phi=P_0 \phi+ P_1\phi = \varphi_0 + \varphi_1, & &
 G*g=P_0(G*g) + P_1(G*g).
\end{align}
Then the equation reduces to
\begin{align}
 \varphi_0 = P_0(G*g), & &
 (I-T_\lm)\varphi_1= P_1(G*g).
\end{align}
Thus~$\varphi_0\in X_k^\perp$ can be directly taken as~$P_0(G*g)$, meanwhile for~$\varphi_1\in X_k^\perp$ we need to be careful: the second equation above is solvable (necessarily uniquely) iff~$P_1(G*g)\in X_k^\perp\cap \mathcal{H}_1$. Note that 
\begin{align}
  <\phi_k, P_1(G*g)>_{H_0^1} 
  =& <P_1(\phi_k), G*g>_{H_0^1} = <\phi_k, G*g>_{H_0^1} \\
  =&\int\limits_\Omega \nabla\phi_k\cdot \nabla (G*g) =\int\limits_\Omega \phi_k (-\Delta(G*g)) 
  =\int\limits_\Omega \phi_k g
\end{align}
which vanishes due to~\eqref{26.7.2-26.3}.
Hence~\eqref{eq:IFT eqn} is uniquely solvable by the Fredholm Alternative.
\end{proof}

In view of Lemma \ref{lem.crux}, by using standard results about the regularity of branches of simple eigenvalues (see for example Proposition 3.6.1. in \cite{BuffoniToland2003analytic}) the rest of the proof is exactly the same as that worked out in \cite{BJ2022uniqueness} and we refer to that paper for further details.
\end{proof}

\begin{proof}[Proof of Theorem~\ref{thm:lm*}]
Recall that~${\mathcal{G}}_{*}$ is the set of solutions for~$\lm\in [0,\lm_*(\mathbb{D}_{_N},p))$, forming a~$C^1$ curve, along which by Theorem C we have
\begin{align}
    \sg_1(\all,\pl)>0, & & \frac{\dd \el}{\dd\lm}>0, & & \frac{\dd\all}{\dd\lm}<0.
\end{align}
Recall also that by Theorem 1.4 in \cite{BJWW2} we already know that $\lm_*(\mathbb{D}_{N},p)\geq \lm_+(\mathbb{D}_{N},p)$.
We argue by contradiction and assume that $\lm_*(\mathbb{D}_{N},p)<+\ii.$

To simplify the exposition we set $\lm_+=\lm_+(\mathbb{D}_{N},p)$ and $\lm_*=\lm_*(\mathbb{D}_{N},p)$.
By Lemma \ref{lemE1}, along a subsequence we can pass to the limit as $\lm_n\to\lm_*(\mathbb{D}_{N},p)$ and deduce that $(\al_n,\psi_n)$
converges in $C^2$ to a solution~$(\al_*,\psi_*) $ of $\prl$ for $\lm=\lm_*$. Remark that
by Theorem B we have $\frac{d \all}{d\lm}<0$ and by assumption $\lm_+\leq\lm_*$, whence by the uniqueness of solutions (Theorem B)
we must have $\all\left.\right|_{\lm=\lm_+}=0$ and consequently $\al_*\leq 0$.

By Remark \ref{rem:var} we see that~$(\al_*,\psi_*)$ is a variational solution and hence by Proposition \ref{prmin} we have $\sg_1(\al_*,\psi_*)\geq 0$.
Then necessarily $\sg_1(\al_*,\psi_*)=0$, as otherwise by Lemma \ref{lem1.1} and the continuity of eigenvalues we would have a
contradiction to the definition of $\lm_*$.
As a consequence, by Proposition~\ref{prop:simple}, we see that the transversality condition needed to apply Proposition
\ref{prop:tranverse} is satisfied, whence we can continue $\mathcal{G}_{*}$ to a $C^1$
parametrization without bifurcation points,
$$
\mathcal{G}_{\lm_*+\eps}=\left\{(-\eps,\eps)\ni s\mapsto (\lm(s),\al(s),\psi(s))\right\},
$$
where, for some $\eps>0$, we have that for any
$s\in (-\eps,\eps)$, $(\al(s),\psi(s))$ is a solution of $\prl$ with $\lm=\lm(s)$ and $(\lm(s),\al(s),\psi(s))\in\mathcal{G}_{*}$ for $s\leq 0$.
Recall from \eqref{2907.1} that $\lm^{'}(0)=0$.

We \emph{claim} that there exists a sequence $s_n\to 0^+$ such that
{$\lm^{'}(s_n)>0$} and the curve $(\lm(s),\al(s),\psi(s))$ bends to the right of $\lm_*$.\\
\noi Indeed, in view of \rife{2907.5} in Proposition \ref{prop:tranverse}, we have,
\begin{align}
    \Abracket{[\phi_1]_{\sscp \lm_*},\psi_*}_{\sscp \lm_*}\neq 0
\mbox{ and } \Abracket{ [\phi_1]_{\sscp \lm_*},\psi_*}_{\sscp \lm_*} \mbox{  has the same sign as }
<\phi_1>_{\sscp \lm_*}\!.
\end{align}
Moreover, again by Proposition \ref{prop:tranverse}, putting $\sg_1(s)=\sg_1(\al(s),\psi(s))$, we have
\begin{align*}
\sg_1(s)\ino (\mbox{\em \Large\emph{\textrho}}_{\! \sscp \lm(s)})^{\frac 1q} [\phi_1(s)]_{\sscp \lm(s)}\psi^{'}(s)=p\lm^{'}(s)\ino
(\mbox{\em \Large\emph{\textrho}}_{\! \sscp \lm(s)})^{\frac 1q} [\phi_1(s)]_{\sscp \lm(s)}\psi(s),
\end{align*}
For $s$ small and negative we have $(\lm(s),\al(s),\psi(s))\in\mathcal{G}_{*}$ whence in particular $\sg_1(s)>0$ and $\lm^{'}(s)>0$, implying that (recall \eqref{2907.1}), since
$\ino (\mbox{\em \Large\emph{\textrho}}_{\! \sscp \lm(s)})^{\frac 1q} [\phi_1(s)]_{\sscp \lm(s)}\psi^{'}(s)=o(1)+\ino (\mbox{\em \Large\emph{\textrho}}_{\! \sscp \lm_*})^{\frac 1q} [\phi_1]^2_{\sscp \lm_*}$, we have
$\Abracket{[\phi_1]_{\sscp \lm_*},\psi_*}_{\sscp \lm_*}>0$. Therefore, we infer for $s$ small and positive that (recall that we have $\sg_1(s)\geq 0$)
$\lm^{'}(s)\geq 0$ and consequently by the mean value theorem
$$
\lm(s)-\lm_*=\lm(s)-\lm(0)=\lm^{'}(t)s\geq 0,
$$
for some $t\in (0,s)$. However we cannot have $\lm(s)=\lm_*$, because this would contradict the uniqueness of solutions at $\lm=\lm_*$.
Therefore $\lm(s)>\lm_*$ for $s$ small and positive, that is, the curve $(\lm(s),\al(s),\psi(s))$ bends to the right of $\lm_*$. In particular
we infer that there exists a sequence $s_n\to 0^+$ such that $\lm^{'}(s_n)>0$ for any $n$, as claimed.

Thus, with the notations of Lemma \ref{lem:al}, putting $\lm_n=\lm(s_n)$, and since $\lm^{'}(s_n)\to 0^+$, we have,
\begin{align}\label{flexal}
    \lim\limits_{n \to +\ii}\left.\frac{d \all}{d\lm}\right|_{\lm=\lm_n}
    =& -\lim\limits_{n \to +\ii}\left.\left(\Abracket{\pl}_{\ssl}+\lm\Abracket{\frac{d \pl}{d\lm}}_{\ssl}\right)\right|_{\lm=\lm_n} \\
    =& \mbox{O}(1)-\lim\limits_{n \to +\ii}\left.\lm\Abracket{\frac{d \pl}{d\lm}}_{\ssl}\right|_{\lm=\lm_n}\\
    =& \mbox{O}(1)-\lm_*\lim\limits_{n\to +\ii}\frac{1}{\lm^{'}(s_n)}\Abracket{\psi^{'}(s_n)}_{\sscp \lm_n} \\
    =& \mbox{O}(1)-\lm_*\lim\limits_{n\to +\ii}\frac{1}{\lm^{'}(s_n)}\left(\Abracket{\phi_1}_{\lm_*}+\mbox{o}(1)\right)=-\ii.
\end{align}
This is impossible, as from Lemma \ref{lem:al} we have that~$\all$ is a global~$C^2$ function of~$\lm\in(0,+\infty)$, which is the desired contradiction.
\end{proof}

\section{On the uniqueness of solutions of the Emden problem}\label{sec:Emden}
In this section we obtain two results related to the uniqueness of solutions of the Emden equation. The first one is a new proof of the uniqueness on $\mathbb{D}_N$ which comes as an immediate consequence of Theorem B. The second one is a uniqueness result which follows from some assumption about
$\lm_*(\om,p)$.
\begin{thm} Let $p\in (1,p_N)$, then \eqref{Emden0} admits a unique solution.
\end{thm}
\begin{proof}
Here $\lm_+=\lm_+(\mathbb{D}_N,p)$. After a suitable scaling it is enough to prove the uniqueness of solutions of \eqref{Emden0} with $B_1\subset \mathbb{R}^N$
replaced by $\mathbb{D}_N$, say \eqref{Emden0}$_N$. Let $(\al_+,\psi_+)$ be the unique solution of $\prl$ on $\mathbb{D}_N$ for $\lm=\lm_+$.
By Theorem B we have $\al_+=0$ and then obviously $u_+=\lm_+^{\frac{p}{p-1}}\psi_+$ is a solution of \eqref{Emden0}$_N$ satisfying
$\int_{\mathbb{D}_N} u^p_+=\lm_+^{\frac{p}{p-1}}$.
If there were any other solution $v$ of \eqref{Emden0}$_N$ distinct from $u_+$, then, defining $\lm_v$ as follows,
$\lm_v^{\frac{p}{p-1}}=\int_{\mathbb{D}_N} v^p$, we would have a solution $(\al_v,\psi_v)$ of $\prl$ with $\al_v=0$ and $\psi_v=\lm_v^{-\frac{p}{p-1}}v$.
If $\lm_v\neq \lm_+$ we have a contradiction to $\frac{d\all}{d\lm}<0$ as stated in Theorem B, while if $\lm_v= \lm_+$ we have a contradiction to the uniqueness.
\end{proof}

\begin{thm} Let $p\in (1,p_N)$, and assume that either $\lm_*(\om,p)=+\ii$ or $\lm_*(\om,p)>\lm_+(\om,p)$ and $\all<0$ for $\lm\geq \lm_*(\om,p)$.
Then the Emden problem
\begin{align*}
    \graf{-\Delta u =u^p\quad \mbox{in}\;\;\om\\
u>0 \quad \mbox{in}\;\;\om \\
u=0 \quad \mbox{on}\;\;\pa \om,
}
\end{align*}
admits a unique solution.
\end{thm}
\begin{proof} Here $\lm_+=\lm_+(\om,p)$, $\lm_*=\lm_*(\om,p)$. Assume first that $\lm_*(\om,p)=+\ii$, then by Theorem C we have that,
for any $\lm\geq 0$, $\prl$ admits a unique solution and the set of solutions is a $C^1$ curve $\lm \to (\all,\pl)$ such that in particular $\frac{d\all}{d\lm}<0$ for any $\lm\geq 0$. Let $(\al_+,\psi_+)$ be the unique solution of $\prl$ on $\om$ for $\lm=\lm_+$, which therefore satisfies
$\al_+=0$. Clearly $\all\neq 0$ for $\lm\neq \lm_+$. Let $u_+=\lm_+^{\frac{p}{p-1}}\psi_+$ be the solution satisfying
$\int_{\om} u^p_+=\lm_+^{\frac{p}{p-1}}$. If there were any other solution $v$ distinct from $u_+$, then, defining $\lm_v$ as follows,
$\lm_v^{\frac{p}{p-1}}=\int_{\om} v^p$, we would have a solution $(\al_v,\psi_v)$ of $\prl$ with $\al_v=0$ and $\psi_v=\lm_v^{-\frac{p}{p-1}}v$.
If $\lm_v\neq \lm_+$ we have a contradiction to $\all\neq 0$ for $\lm\neq \lm_+$, while if $\lm_v= \lm_+$ we have a contradiction to the uniqueness.\\
Assume now that $\lm_*(\om,p)>\lm_+(\om,p)$ and $\all<0$ for $\lm\geq \lm_*(\om,p)$. The argument above works fine in this case as well, the only difference
being that $\all\neq 0$ for $\lm\neq \lm_+$ is deduced from $\frac{d\all}{d\lm}<0$ for $\lm<\lm_*$ and by the assumption $\all<0$ for $\lm\geq \lm_*(\om,p)$ otherwise.
\end{proof}

\bibliographystyle{siam}
\bibliography{Plasma}

\end{document}